\documentclass{article}[11pt]

\usepackage{tikz}
\usepackage[ruled,vlined,linesnumbered]{algorithm2e}
\usepackage{fullpage}
\usepackage{cite}
\usepackage{graphicx}
\usepackage{amsmath,amsfonts,amssymb,latexsym}
\usepackage{enumerate}
\usepackage{setspace}
\usepackage{color}
\usepackage[utf8]{inputenc}

\tikzstyle{vertex}=[circle, draw, inner sep=0pt, minimum size=6pt]

\newenvironment{claimproof}[1]{{\it\noindent{Proof.}}\space#1}{\footnotesize \hfill \ensuremath{(\square)} \medskip}
\newcommand{\qed}{\hfill ~$\square$\bigskip}
\newtheorem{thm}{Theorem}[section]
\newtheorem{prop}[thm]{Proposition}

\newtheorem{cor}[thm]{Corollary}

\newtheorem{lemma}[thm]{Lemma}
\newtheorem{claim}{Claim}

\newtheorem{prob}{Problem}

\def\NN{\hbox{\sf I\kern-.13em\hbox{N}}}
\def\RR{\hbox{\sf I\kern-.14em\hbox{R}}}
\def\ZZ{\hbox{\sf I\kern-.14em\hbox{Z}}}

\begin{document}

\title{On the dissociation number of Kneser graphs}

\author{
Bo\v{s}tjan Bre\v{s}ar$^{a,b}$\thanks{bostjan.bresar@um.si}
 \\
Tanja Dravec$^{a,b}$\thanks{tanja.dravec@um.si}
\\
}


\maketitle

\begin{center}
$^a$ Faculty of Natural Sciences and Mathematics, University of Maribor, Slovenia\\
$^b$ Institute of Mathematics, Physics and Mechanics, Ljubljana, Slovenia\\

\end{center}

\begin{abstract}
A set $D$ of vertices of a graph $G$ is a dissociation set if each vertex of $D$ has at most one neighbor in $D$. The dissociation number of $G$, $diss(G)$, is the cardinality of a maximum dissociation set in a graph $G$. In this paper we study dissociation in the well-known class of Kneser graphs $K_{n,k}$. In particular, we establish that the dissociation number of Kneser graphs $K_{n,2}$ equals $\max{\{n-1,6\}}$. We show that for any $k \geq 2$, there exists $n_0 \in \mathbb{N}$ such that $diss(K_{n,k})=\alpha(K_{n,k})$ for any $n \geq n_0$. We consider the case $k=3$ in more details and prove that $n_0=8$ in this case. Then we improve a trivial upper bound $2\alpha(K_{n,k})$ for the dissociation number of Kneser graphs $K_{n,k}$ by using Katona's cyclic arrangement of integers from $\{1,\ldots , n\}$. Finally we investigate the odd graphs, that is, the Kneser graphs with $n=2k+1$. We prove that $diss(K_{2k+1,k})={2k \choose k}$.    

\end{abstract}


\noindent
{\bf AMS Subj.\ Class.\ (2010)}: 05C69 

\section{Introduction}

The {\em Kneser graph}, $K_{n,k}$, where $n,k$ are positive integers such that $n\ge 2k$,  has the $k$-subsets of an $n$-set as its vertices, and two $k$-subsets are adjacent in $K_{n,k}$ if they are disjoint. The Erd\H{o}s--Ko--Rado theorem~\cite{ekr-61} determined the independence number $\alpha(K_{n,k})$ of the Kneser graph $K_{n,k}$ to be equal to ${{n-1}\choose{k-1}}$. Another famous result is Lov\'{a}sz's proof of Kneser's conjecture, which determines the chromatic number of Kneser graphs~\cite{lov-78}, see also Matou\v{s}ek for a combinatorial proof of this result~\cite{mat-04}. Many other invariants were later considered in Kneser graphs by a number of authors. The diameter of a Kneser graph $K_{n,k}$ was computed in~\cite{vpv-05} and the hamiltonicity was researched in~\cite{Chen2003,Si2004}. The domination number of Kneser graphs was also studied in several papers~\cite{hw-2003,iz-1993,osx-2015}, but there is no such complete solution for domination number of Kneser graphs as is the case with the chromatic and the independence number. Recently, the $P_3$-hull number of Kneser graphs was completely resolved for all Kneser graphs $K_{n,k}$ with the sole exception of odd graphs, that is, when $n=2k+1$; see~\cite{grip-21}.  The problem of independence in graphs can be rephrased as the search for a (largest) induced subgraph in which all components have only one vertex. 
In this paper, we extend the study to search for a largest induced subgraph of a Kneser graph in which all components have at most two vertices.  

A set $D$ of vertices in a graph $G$ is called a {\em dissociation set} if the subgraph induced by vertices of $D$ has maximum degree at most 1.  The cardinality of a maximum dissociation set $D$ in a graph $G$ is called the {\em dissociation number} of $G$, and is denoted by $diss(G)$. The dissociation number was introduced by Papadimitriou and Yannakakis~\cite{py-82} in relation with the complexity of the so-called restricted spanning tree problem. A dual concept to dissociation set can be generalized to {\em $m$-path vertex cover}, which was introduced in~\cite{bkk-11} and studied in several papers~\cite{bs-16,bjk-13,jt-13}; it is defined as a set $S$ of vertices in $G$ such that $G-S$ does not contain any path $P_m$. The corresponding invariant, the {\em $m$-path vertex cover number} of an arbitrary graph $G$, is denoted by $\psi_m(G)$. 
Note that dissociation sets are complements of $3$-path vertex covers of $G$, and so $diss(G)=|V(G)|-\psi_3(G)$. The decision version of the $m$-path vertex cover number is NP-complete~\cite{bkk-11}, moreover, in the case $m=3$ it is NP-complete even in bipartite graphs which are $C_4$-free and have maximum degree 3~\cite{bcl-04}; see also~\cite{odf-11} for further strengthening of this result and~\cite{kks-11} for an approximation algorithm. Some variations of the problem were already studied as well (see e.g. \cite{bkss,liyawa-17}). We mention in passing that graphs in which all maximal dissociation sets are of the same size were studied in~\cite{bhr-17}.

The independence number of a graph $G$, $\alpha(G)$, can be defined as the order of the largest induced subgraph of $G$ with maximum degree 0. If 0 in this definition is replaced by 1, then we get a definition of the dissociation number of a graphs. Since any independent set of a graph $G$ is also a dissociation set of $G$, the independence number of $G$ is a lower bound for the dissociation number of $G$. In addition, one can easily get the upper bound for the dissociation number of $G$ as a function of $\alpha(G)$. Let $S$ be a dissociation set of $G$ and $A \subseteq S$ a maximum independent subset of $S$. Then every vertex of $S \setminus A$ has exactly one neighbor in $A$ and any vertex of $A$ has at most one neighbor in $S \setminus A$. Therefore $|S \setminus A| \leq |A|$. Hence we immediately get the following bounds for the dissociation number of $G$:
$$\alpha(G) \leq diss(G) \leq 2\alpha(G).$$

The paper is organized as follows. In Section~\ref{sec:dissociation}, we first present the exact result for the dissociation number of Kneser graphs $K_{n,2}$. Then we prove that for any $k \geq 2$ there exists $n_0 \in \NN$ such that $diss(K_{n,k})=\alpha(K_{n,k})$ for any $n \geq n_0$. Also, we find $n_0$ for $k=3$; we prove that $diss(K_{n,3})=\alpha(K_{n,3})$ if and only if $n \geq 8$. In Section~\ref{sec:bound}, we use Katona's cyclic arrangement of integers from his proof of Erd\H{o}s-Ko-Rado theorem~\cite{kat-72} to improve the upper bound $2 \alpha(G)$ for the dissociation number for the case when $G$ is a Kneser graph.
In Section~\ref{sec:odd}, we show that the dissociation number of odd graphs $O_k$ (Kneser graphs $K_{2k+1,k}$) equals ${2k \choose k}$ for $k\geq 2.$ 

In the rest of this section we present the notation used throughout the paper and some basic results concerning the dissociation number of a graph. 

Let $[n] = \{1,2,\ldots,n\}$, where $n\in \NN$. For a graph $G=(V,E)$ and $S \subseteq V(G)$ we write $G[S]$ for the subgraph of $G$ induced by $S$ and $G -S$ for the subgraph of $G$ induced by the set $V(G) \setminus S$. 
The {\em (open) neighborhood} of $v \in V(G)$, $N_G(v)$, is the set of all neighbors of $v$, while $N_G[v]=N_G(v) \cup \{v\}$ denotes the {\em closed neighborhood} of $v$. Similarly, for $S \subseteq V(G)$, $N_G[S]= \bigcup_{v \in S}N_G[v]$. The {\em degree} of a vertex $v$ is $|N_G(v)|$. When the graph $G$ is clear from the context we omit the subscripts. A {\em matching} $M$ in a graph $G$ is a set of edges in $G$ having the property that no two
edges in $M$ have a common endvertex. For a matching $M$ in $G$, we denote by $V(M)$ the set of endvertices of edges from $M$. A set of pairwise non adjacent vertices in a graph $G$ is called the {\em independent set}. The cardinality of the largest independent set of vertices in $G$ is the {\em independence number} of $G$ and is denoted by $\alpha(G)$.

A {\em center} of a Kneser graph $K_{n,k}$ is a set $I(i)=\{x\in V(K_{n,k}):\, i\in x\}$, where $i\in [n]$. Note that $I(i)$ is an independent set of vertices of $K_{n,k}$, and $|I(i)|=\alpha(K_{n,k})={{n-1}\choose {k-1}}$. Note that for any $n \geq 2k$, $K_{2k,k}$ is an induced subgraph of $K_{n,k}$ and hence $diss(K_{n,k}) \geq diss(K_{2k,k})=|V(K_{2k,k})|={2k \choose k}$. We state this as follows.

\begin{prop}\label{prp:lowerBound}
For any $n \geq 2k$,  $diss(K_{n,k}) \geq {2k \choose k}$.
\end{prop}

\section{Relations  with the independence number} 
\label{sec:dissociation}

We start the study with the simplest non-trivial Kneser graphs, that is, $K_{n,2}$, where $n\ge 5$. For the Petersen graph $K_{5,2}$, one can easily see that $D=\{\{1,2\},\{3,4\},\{1,3\},\{1,4\},\{2,3\},\{2,4\}\}$ is a dissociation set. Indeed, $D$ induces a subgraph with only three edges, namely $\{1,2\}\{3,4\}$,$\{1,3\}\{2,4\}$ and $\{1,4\}\{2,3\}$. It is also a largest dissociation set, hence $diss(K_{5,2})=6$. The same construction is optimal also in $K_{6,2}$ and $K_{7,2}$, but not for $K_{n,2}$, with larger $n$, as the following result shows.

\begin{prop} 
\label{prp:Kn2}
For $n\ge 5$, $diss(K_{n,2})=\max\{n-1,6\}$. 
\end{prop}
\begin{proof}
Note that a maximum independent set of a Kneser graphs $K_{n,2}$, where $n>5$, is a center, its size is $n-1$, and it is also a dissociation set. By the above observation, we get $diss(K_{n,2})\ge \max\{n-1,6\}$.
Suppose that $D$ is a dissociation set, which is not independent. Without loss of generality, let $\{\{1,2\},\{3,4\}\}\subset D$. Note that $V(K_{n,2})\setminus N[D]=\{\{1,3\},\{1,4\},\{2,3\},\{2,4\}\}$, which implies $|D|\le 6$. Hence, if $n> 7$, a maximum dissociation set is independent, and the proposed equality follows. 
\qed
\end{proof}

We find a similar feature for Kneser graphs $K_{n,k}$, where $k>2$. Namely, as soon as $n$ is large enough with respect to $k$, we have $diss(K_{n,k})=\alpha(K_{n,k})$. 

\begin{thm}
\label{thm:large-n}
For any $k\ge 2$, there exists $n_0\in\NN$ such that for all $n$, $n\ge n_0$, we have $$diss(K_{n,k})=\alpha(K_{n,k})={{n-1}\choose{k-1}}.$$
\end{thm}
\begin{proof}
The result for $k=2$ follows from Proposition~\ref{prp:Kn2}. 
Fix $k\ge 3$, and suppose that a maximum dissociation set $D$ is not an independent set. Assume without loss of generality that $x=\{1,\ldots,k\}$ and $y=\{k+1,\ldots,2k\}$ belong to $D$.  Let $U=V(K_{n,k})\setminus N[\{x,y\}]$. 
Note that every element in $U$ is a $k$-set that contains at least one element from $x$ and at least one element from $y$. Setting $z=\{2k+1,\ldots,n\}$ note that any element in $U$ consists of $i$ elements from $z$, where $0\le i\le k-2$, $j$ elements from $x$, where $1\le j\le k-i-1$, and consequently, $k-i-j$ elements from $y$. 

Hence 
$$|U|=\sum_{i=0}^{k-2}{{n-2k\choose i}}\sum_{j=1}^{k-i-1}{{k\choose j}{k\choose k-j-i}}.$$
Note that $\sum_{j=1}^{k-i-1}{{k\choose j}{k\choose k-j-i}}$ is not dependent on $n$, hence for a fixed $k$ this is a constant, while $\sum_{i=0}^{k-2}{{n-2k\choose i}}$ is a polynomial in $n$ of degree $k-2$. Hence $|U|=O(n^{k-2})$, and note that $|D|\le 2+|U|$.  On the other hand, $\alpha(K_{n,k})={n-1\choose k-1}$, hence the resulting independent (and dissociation) set is of size $\Omega(n^{k-1})$. 
Therefore, if $n$ is big enough, $D$ is not a maximum dissociation set, because its size is less than ${n-1\choose k-1}$.
\qed
\end{proof}

Note that the above proof relies on the fact that for any adjacent vertices $x$ and $y$ of a dissociation set $D$ of $G$, we have $D\subseteq D_{x,y}'=\{x,y\}\cup (V(G)\setminus (N[x]\cup N[y]))$. In fact, we show that for any $k\ge 2$ there exists $n_0' \in \mathbb{N}$ such that for any adjacent vertices $x,y \in V(K_{n,k})$ we have $\alpha(K_{n,k}) \geq |D'_{x,y}|$ as soon as $n \geq n_0'$.
In particular, the smallest $n_0$ that appears in the statement of Theorem~\ref{thm:large-n} may be much smaller than $n_0'$ which is used in the proof.
Note that $|D_{x,y}'|=2+ {n \choose k}-(2{n-k \choose k}-{n-2k \choose k})$. For $k=2$ the smallest $n_0'$, for which $\alpha(K_{n,2}) \geq |D_{x,y}'|$ when $n\ge n_0'$, is $7$, which is by Proposition~\ref{prp:Kn2} also $n_0$ from Theorem~\ref{thm:large-n} (that is, $diss(K_{n,2})=\alpha(K_{n,2})$ as soon as $n \geq n_0=7$). This is not the case for $k>2$. For $k=3$ one can easily compute that $n_0'=17$ (by solving inequality $\alpha(K_{n,k}) \geq |D_{x,y}'|$ for $k=3$), but as we will see in Corollary~\ref{cor:n_0 for k=3}, we have $diss(K_{n,3})=\alpha(K_{n,3})$ already for $n \geq 8$. For $k > 3$, we do not know how large must $n_0$ be in Theorem~\ref{thm:large-n} and we propose this as an open problem.

\begin{prob}
\label{prob:n_0}
Given an integer $k\ge 4$, what is the smallest integer $n_0$ such that for all $n\ge n_0$, $diss(K_{n,k})=\alpha(K_{n,k})$?
\end{prob}

From Proposition~\ref{prp:lowerBound} we get the inequality that leads to the lower bound for $n_0$ from Theorem~\ref{thm:large-n}. 

\begin{lemma}\label{l:k=3LowerBound} 
The smallest integer $n_0$ for which $diss(K_{n_0,k})=\alpha(K_{n_0,k})$, is at least $2k+2$.
\end{lemma}
\begin{proof}
Let  $n \geq 2k$ and $diss(K_{n,k})=\alpha(K_{n,k})$. Then ${n-1 \choose k-1} \geq {2k \choose k}$ by Proposition~\ref{prp:lowerBound}. Solving this inequality we infer $n \geq 2k+2$. \qed
\end{proof}

We follow with establishing the exact value of $diss(K_{8,3})$.

\begin{lemma}\label{lema:K_8,3}
$diss(K_{8,3})=\alpha(K_{8,3})$.
\end{lemma}
\begin{proof}
For the purpose of contradiction assume that $diss(K_{8,3}) > 21$. Let $S$ be a maximum dissociation set. Since $|S| > \alpha(K_{8,3})$, $S$ is not independent. Without loss of generality we may assume that $x=\{1,2,3\},y=\{4,5,6\}$ are two adjacent vertices contained in $S$.  Since $S$ is a dissociation set, $S \cap (N[x] \cup N[y])=\{x,y\}$. Let $H$ be the subgraph of $K_{8,3}$ induced by $V(K_{8,3})\setminus (N[x] \cup N[y])$. We define the following sets
\begin{itemize}
\item $U=\{z \in V(H); 7 \in z\}$, $U'=\{z \in V(H); 8 \in z\}$;
\item $\forall i \in [6]$, $U_i=\{z \in U; i \in z\}$, $U'_i=\{z \in U'; i \in z\}$;
\item $D=V(H)\setminus (U\cup U')$;
\item $D'=\{z \in D; |z \cap x|=2\}$, $D''=\{z \in D; |z \cap y|=2\}$;
\item $\forall i,j \in [3],~D_{ij}'=\{z \in D'; i,j \in z\}$;
\item $\forall i,j \in \{4,5,6\}, ~D_{ij}''=\{z \in D''; i,j \in z\}$;
\item $\forall i \in \{4,5,6\},~D_{i}'=\{z \in D'; i \in z\}$;
\item $\forall i \in [3]$ $D_{i}''=\{z \in D''; i \in z\}$.
\end{itemize}
If $z \in V(K_{8,3})$ with $\{7,8\} \subseteq z$, then $z \in N[x] \cup N[y]$. Thus $U \cap U' = \emptyset$, and so $[U,U',D]$ is a partition of $V(H)$.
Also $[U_1,U_2,U_3]$ and $[U_4,U_5,U_6]$ are partitions of $U$, $[U_1',U_2',U_3'], [U_4',U_5',U_6']$ are partitions of $U'$, $[D_{12}',D_{13}',D_{23}'], [D_4',D_5',D_6']$ are partitions of $D'$ and $[D_{45}'',D_{46}'',D_{56}''], [D_1'',D_2'',D_3'']$ are partitions of $D''$. Moreover, sets $U,U',D',D''$ are independent sets of cardinality 9. A spanning subgraph $H'$ of a graph $H$ is depicted in Figure~\ref{fig:H}. The edges of $H$ that are not in $H'$ are the edges between $U$ and $U'$, between $U$ and $D''$, between $D'$ and $D''$. Note that $H[U \cup D'] \cong H[U \cup D''] \cong H[U' \cup D'']$.

\begin{figure}[!ht]
\begin{center}
\includegraphics[width=10cm]{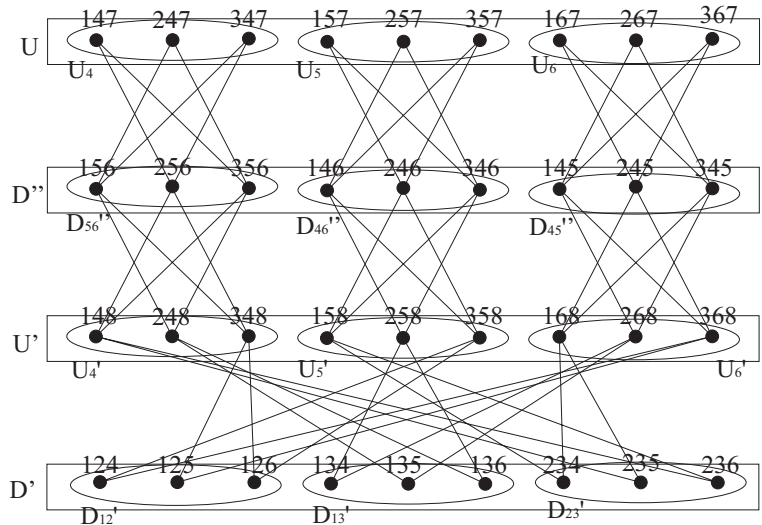}
\caption{Spanning subgraph of a graph $H$.}
\label{fig:H}
\end{center}
\end{figure}

\begin{claim}\label{c:1}
$|U_i \cap S| < 3$ for any $i \in \{4,5,6\}$.
\end{claim}
\begin{claimproof}
Suppose that there exists $i \in \{4,5,6\}$ such that $|U_i \cap S|=3$ and let $\{i,j,k\}=\{4,5,6\}$. Hence $(U_j' \cup U_k' \cup D_{jk}'') \cap S=\emptyset$, as each vertex of $U_j' \cup U_k' \cup D_{jk}''$ has exactly 2 neighbors in $U_i$. Let $\ell$ be an arbitrary element of $\{1,2,3\}$ and let $\{\ell,\ell',\ell''\}=\{1,2,3\}$. Since $\{\ell',\ell'',j\},\{\ell',\ell'',k\}$ are both neighbors of $\{l,i,7\} \in U_i \subseteq S$, $|\{ \{\ell',\ell'',j\},\{\ell',\ell'',k\}\} \cap S| \leq 1$. Note that subgraphs of $H$ induced by $U_j \cup D_{ik}''$ and $U_k \cup D_{ij}''$ are isomorphic to $C_6$ and thus $|(U_j \cup D_{ik}'') \cap S|, |(U_k \cup D_{ij}'') \cap S|  \leq 4$.
If $|(U_j \cup D_{ik}'') \cap S| = 4$, then $S$ contains exactly 2 vertices from $U_j$ and thus it contains at most 2 verices from $U_i'$ (as $H[U_j \cup U_i'] \cong C_6$). Therefore $|(U_j \cup D_{ik}''\cup U_i') \cap S| \leq 6$. Hence,

\noindent $|S \cap V(H)|=|S \cap U_i|+|S \cap (U_j' \cup U_k' \cup D_{jk}'')|+ |S \cap (U_j \cup D_{ik}'' \cup U_i')|+ |S \cap (U_k \cup D_{ij}'')|+|S \cap D_{12}'|+|S \cap D_{13}'|+|S \cap D_{23}'| \leq 3+0+6+4+2+2+2=19. $
%
We infer $|S| \leq 21$, a contradiction.
\end{claimproof} 

Analogous arguments imply

\begin{claim}\label{c:2}
For any $i \in \{4,5,6\}$  $|U_i' \cap S| < 3$.
\end{claim}

\begin{claim}\label{c:3}
For any $i \in [3]$ let $\{i,j,k\}=\{1,2,3\}$. Then $D_{jk}' \cap S \neq \emptyset$.
\end{claim}

\begin{claimproof}
Suppose that there exists $i \in \{1,2,3\}$, such that $|D_{jk}' \cap S|=0$, where $\{j,k\}=\{1,2,3\} \setminus \{i\}$. For $\ell \in \{4,5,6\}$ and $\{\ell',\ell''\} \in \{4,5,6\}\setminus \{\ell\}$ let $A_{\ell}$ be the subgraph of $H$ induced by $U_{\ell} \cup D_{\ell'\ell''}'' \cup U_{\ell'}$. Note that the subgraphs of $A_{\ell}$ induced by  $U_{\ell} \cup D_{\ell'\ell''}''$ and by  $D_{\ell'\ell''}'' \cup U_{\ell'}$ are isomorphic to $C_6$. 

Suppose that $|A_{\ell} \cap S| > 4$. Then $D_{\ell'\ell''}'' \cap S \neq \emptyset$, as $|U_{\ell} \cap S| \leq 2$ and $|U_{\ell}' \cap S| \leq 2$ by Claims~\ref{c:1} and \ref{c:2}. Without loss of generality we may assume that $\{1,\ell',\ell''\} \in S \cap D_{\ell' \ell''}''$. Since $\{2,\ell,8\},\{3,\ell,8\},\{2,\ell,7\},\{3,\ell,7\}$ are all neighbors of $\{1,\ell',\ell''\} \in S$, we get  $|\{ \{2,\ell,8\},\{3,\ell,8\},\{2,\ell,7\},\{3,\ell,7\}\} \cap S| \leq 1$. Since $\{1,\ell,8\},\{1,\ell ,7\}$ are both neighbors of $\{2,\ell',\ell''\}$ and $\{3,\ell',\ell''\}$, the following statements hold.
\begin{itemize}
\item If $\{\{2, \ell',\ell''\},\{3,\ell',\ell''\}\} \subseteq S$, then $S \cap \{\{1,\ell,8\},\{1,\ell ,7\}\} = \emptyset $;
\item If $|S \cap \{\{1,\ell,8\},\{1,\ell ,7\}\}|=2$, then $S \cap \{\{2, \ell',\ell''\},\{3,\ell',\ell''\}\} = \emptyset$.
\end{itemize}
Thus $S$ contains at most two vertices from $\{\{2, \ell',\ell''\},\{3,\ell',\ell''\},\{1,\ell,8\},\{1,\ell ,7\}\}$ and consequently $|A_{\ell} \cap S| \leq 4$, a contradiction. Hence $|A_{\ell} \cap S| \leq 4$ for any $\ell \in \{4,5,6\}$. Since $D_{jk}' \cap S = \emptyset$, $|S \cap D'| \leq 6$ and thus $$|V(H) \cap S| = \sum_{l=4}^6 |A_l \cap S|+ |D' \cap S| \leq 3\cdot 4 +6=18,$$ a contradiction. 
\end{claimproof}

For any $i \in [3]$ define $X_i''$ as a subgraph of $H$ induced by vertices of $D_i'' \cup U_i \cup U_i' \cup  D_{jk}'$, where $\{j,k\}=[3]\setminus \{i\}$.

Let $i \in [3]$ and $\{i,j,k\}=\{1,2,3\}$. Since by Claim~\ref{c:3} $D_{jk}' \cap S \neq \emptyset$, we have three possibilities.

If $|D_{jk}' \cap S|=1$. Let $\ell \in \{4,5,6\}$ such that $\{j,k,\ell\} \in S$ and let $\{4,5,6\}\setminus\{\ell\}=\{\ell',\ell''\}$. Since $\{j,k,\ell \}$ is adjacent to all vertices from $A=\{\{i,\ell',\ell''\},\{i,\ell',7\},\{i,\ell'',7\},\{i,\ell',8\},\{i,\ell'',8\}\}$, $|S \cap A| \leq 1$. Since $|X_i''\setminus(A \cup D_{jk}')|=4$, $|S \cap A| \leq 1$ and $|S \cap D_{jk}'|=1$ we get $|S \cap V(X_i'')| \leq 6$.

If $|D_{jk}' \cap S|=2$, let $\ell,\ell' \in \{4,5,6\}$ such that $\{j,k,\ell\},\{j,k,\ell'\} \in S$ and let $\{4,5,6\} \setminus \{\ell,\ell'\}=\{\ell''\}$. Since $\{i,\ell'',7\},\{i,\ell'',8\}$ are adjacent to both vertices from $D_{jk}' \cap S$,  $\{i,\ell'',7\},\{i,\ell'',8\} \notin S$. Since $A=\{\{i,\ell',\ell''\},\{i,\ell',7\},\{i,\ell',8\}\}$ is the set of neighbors in $X_i''$ of $\{j,k,\ell\} \in D_{jk}' \cap S$, at most one vertex from $A$ can be contained in $S$. Similarly at most one vertex from $B=\{\{i,\ell,\ell''\},\{i,\ell,7\},\{i,\ell,8\}\}$ can be contained in $S$, as all vertices of $B$ are neighbors of $\{j,k,\ell'\}$. Since $|X_i''-(A \cup B \cup D_{jk}')|=1$, $|S \cap V(X_i'')| \leq |S \cap D_{jk}'|+|S \cap A|+ |S \cap B|+1 \leq 2+1+1+1=5$.

Finally, let $S \cap D_{jk}' =D_{jk}'$. Then $U_i \cap S = U_i' \cap S = \emptyset$ and thus $|V(X_i'') \cap S| \leq 6$.

\vspace{2mm}

We have proved that for any $i \in [3]$ it holds $|V(X_i'') \cap S| \leq 6$. Therefore $$|S \cap V(H)|=|S \cap V(X_1'')|+|S \cap V(X_2'')|+|S \cap V(X_3'')| \leq 18,$$ which is a final contradiction.   
\qed
\end{proof}

\begin{prop}\label{p:d=a,k=3}
Let $n \geq 9$. If $diss(K_{n-1,3})=\alpha(K_{n-1,3})$ then $diss(K_{n,3})=\alpha(K_{n,3}).$
\end{prop}
\begin{proof}
Suppose that there exists $n\geq 9$ with $diss(K_{n-1,3})=\alpha(K_{n-1,3})$ and $diss(K_{n,3}) > \alpha(K_{n,3})= {n-1 \choose 2}$. Let $S$ be a maximum dissociation set in $K_{n,3}$. Since $|S| > \alpha(K_{n,3})$, $S$ is not an independent set. Note that since $n\geq 9$, $|S| \geq {n-1 \choose 2}+1 \geq 29$. Without loss of generality we may assume that $x=\{1,2,3\}, y=\{4,5,6\} \in S$. Since $S$ is a dissociation set, $(S \cap (N[x] \cup N[y])) \setminus \{x,y\} = \emptyset$. Let $H$ be the subgraph of $K_{n,3}$ induced by $V(K_{n,3}) \setminus (N[x] \cup N[y]).$ Let $U=\{z \in V(H);~ n \in z\}$ and $D=\{z \in V(H);~n \notin z\}$. Since each vertex $z  \in U$ contains exactly one element from $x$ and exactly one element from $y$, $|U|=9$. 

If $D \cap S = \emptyset$, then $|S| \leq 2+9=11$, a contradiction. Hence we may assume that $D \cap S \neq \emptyset$. Let $z \in D \cap S$. Then at least one element from $x$, say $i$, and at least one element from $y$, say $\ell$, is contained in $z$. Denote $z=\{i,\ell,w\}$.

Suppose first that $w \notin x \cup y$. Let $\{1,2,3\}=\{i,j,k\}$, $\{4,5,6\}=\{\ell, \ell', \ell''\}$. Then all vertices of  $A=\{\{j,\ell',n\},\{j,\ell'',n\},\{k,\ell',n\},\{k,\ell'',n\}\} \subseteq U$ are neighbors of $z \in S$. Hence $|S \cap A| \leq 1$ and therefore $|S \cap U| \leq 6.$ Since vertices of $N[x] \cup N[y]$ that do not contain $n$ together with $D$ induce $K_{n-1,3}$, $|S| \leq \alpha(K_{n-1,3}) +|U \cap S| \leq {n-2 \choose 2}+6$. Hence for any $n \geq 8$ we get $|S| \leq  \alpha(K_{n,3})$, a contradiction. 

Hence $w \in x \cup y$, or, in other words, $S$ does not contain vertices $\{i,j,z\}$, where $i \in \{1,2,3\}, j\in \{4,5,6\}, z \in \{7,\ldots , n-1\}$. Thus if $d \in S \cap D$, then $d=\{i_1,i_2,\ell\}$ or $d=\{i,\ell_1,\ell_2\}$, where $i,i_1,i_2\in \{1,2,3\}$, $\ell,\ell_1,\ell_2 \in \{4,5,6\}$. Hence $|S \cap D| \leq 18$. If $S$ contains $\{i_1,i_2,\ell\}$ for  $\{i,i_1,i_2\}= \{1,2,3\}$, $\{\ell,\ell_1,\ell_2\} = \{4,5,6\}$, then $S$ cannot contain both $\{i,\ell_1,n\},\{i,\ell_2,n\} \in U$ and thus $|S \cap (U \cup D)| \leq 26$. Therefore $|S| \leq 28$, a contradiction.  \qed
\end{proof}

We suspect that Proposition~\ref{p:d=a,k=3} also holds for $k$ bigger than $3$, and pose it as a problem. 

\begin{prob} Is it true that $diss(K_{n,k})=\alpha(K_{n,k})$ implies $diss(K_{n+1,k})=\alpha(K_{n+1,k})$ for all $k\ge 2$ and $n\ge 2k+2$?
\end{prob}

As a direct corollary of Lemma~\ref{l:k=3LowerBound}, Lemma~\ref{lema:K_8,3} and Proposition~\ref{p:d=a,k=3} we get that $diss(K_{n,3})=\alpha(K_{n,3})$ for any $n \geq 8$.

\begin{cor}\label{cor:n_0 for k=3}
For $k=3$, $diss(K_{n,k})=\alpha(K_{n,k})$ if and only if $n \geq 8$.
\end{cor}

\section{An upper bound for $diss(K_{n,k})$}
\label{sec:bound}

In this section, we consider upper bounds for the dissociation number of Kneser graphs $K_{n,k}$. We already know that $\alpha(K_{n,k}) \leq diss(K_{n,k}) \leq 2\alpha(K_{n,k})$ and that if $n$ is large enough, then the dissociation number coincide with the independence number. Thus the bound is interesting only when $n<n_0$, where $n_0$ is the integer that appears in Theorem~\ref{thm:large-n} and Problem~\ref{prob:n_0}. Our aim of this section is to improve the upper bound $2 \alpha(K_{n,r})$ for $n < n_0$.

We will improve the upper bound by using Katona's cyclic arrangement of integers from his proof of Erd\H{o}s-Ko-Rado theorem. Let $\cal D$ be a maximum dissociation set of $K_{n,k}$. We count in two different ways the number of ordered pairs $(D,C)$, where $D\in \cal D$ and $C$ is a cyclic arrangement of integers from $[n]$ in which $D$ appears as a substring. 

Let $n=2k+r$, thus we consider $K_{2k+r,k}$, where $r\ge 1$.
If one takes any set from $\cal D$, then it appears as a substring in $k!(n-k)!$ different cyclic arrangements. Thus, altogether there are $|{\cal D}|k!(n-k)!$ such ordered pairs $(D,C)$. Second, note that there are $(n-1)!$ different cyclic arrangements of integers from $[n]$. Next, let us bound from above the number of sets from $\cal D$ that appear as substrings in any given cyclic arrangement. 


We distinguish two cases. First, let $r>k-2$. We claim that in any given cyclic arrangement there are at most $k+1$ elements from $\cal D$ that appear as its substrings. Suppose that all elements of $\cal D$ that appear as substring in $C$ pairwise intersect. Then, it is easy to see that at most $k$ elements from $\cal D$ appear as substrings in $C$.  Without loss of generality, let $D_1:1,2,\ldots, k$, and $D_2:k+t,k+t+1,2k+t-1$, $t\in [n-2k+1]$, be the substrings in $C$ that correspond to elements of $\cal D$. Since the sets that correspond to $D_1$ and $D_2$ form an edge in $K_{n,r}$, we infer that all other sets of $\cal D$ that appear as substrings in $C$ must intersect both $D_1$ and $D_2$. Since $r>k-2$, we infer that there are at most $k-1$ such substrings of length $k$ that intersect both $D_1$ and $D_2$. This implies that there are at most $k+1$ elements from $\cal D$ that appear as substrings, as claimed. Thus, when $n> 3k-2$, we get 
\begin{equation}
\label{eq:rbig}
diss(K_{n,k})\le \frac{k+1}{k}{n-1\choose k-1}\,.
\end{equation}

The second case is $r\le k-2$. Again, let $D_1:1,2,\ldots, k$ be a substring in a cyclic arrangement $C$ that corresponds to an element of $\cal D$ (by abuse of language, we denote this element by $D_1$ as well). Clearly, there is at most one set in $\cal D$ that does not intersect $D_1$. 
Note that a set $D$ in ${\cal D} \setminus \{D_1\}$ can intersect $D_1$ in two different ways, either $1\in D$ or $k\in D$. We denote by $A_i$ the substring in $C$ for which $A_i\cap D_1=[i]$, and by $B_i$ the substring in $C$ for which $B_i\cap D_1=[k] \setminus [i]$. If for some $i\in [n-1]$, $A_i\in \cal D$ and $B_i\in \cal D$, then $i$ is a {\em double point}. On the other hand, if just one of $A_i\in \cal D$ or $B_i\in \cal D$ holds, then $i$ is a {\em single point}. Let $d$ be the number of double points and $s$ the number of single points. Note that the number of substrings of $C$ that correspond to elements of $\cal D$ is bounded from above by $2+s+2d$, where 2 coresponds to $D_1$ and possibly one more element from $\cal D$ that does not intersect $D_1$. 

Suppose that $i$ is a double point. Since $\cal D$ is a dissociation set, any set in ${\cal D}\setminus \{A_i,B_i\}$ must intersect both $A_i$ and $B_i$. This in turn implies that $A_{i-1},\ldots,A_{\max\{i-r,1\}}$ do not belong to $\cal D$ and also $B_{i+1},\ldots,B_{\max\{i+r,k\}}$ do not belong to $\cal D$. In other words, a double point can appear at most in every $2r+1$ turn, that is, at most $\lceil \frac{k}{2r+1} \rceil$ times. Hence,  the number of substrings of $C$ that correspond to elements of $\cal D$ is bounded from above by $2+s+2d=2+k-1+\lceil \frac{k}{2r+1} \rceil \leq 2+k+\frac{k}{2r+1}$.  This yields
$$|{\cal D}|(n-k)!k! \le (2+k+\frac{k}{2r+1})(n-1)!\,$$
which implies 
\begin{equation}
\label{eq:second}
diss(K_{2k+r,k})\le 2\frac{rk+2r+k+1}{k(2r+1)}{n-1\choose k-1}\,.
\end{equation}



\section{Dissociation number of odd graphs}
\label{sec:odd}

In this section, we use the famous Hall's marriage theorem, which we next formulate. Let $G$ is a bipartite graph, where a bipartition of $V(G)$ is $[X,Y]$. A matching $M$ in $G$ is an {\em $X$-matching} of $G$ if every vertex in $X$ is incident with an edge of $M$. 
\begin{thm}
\label{thm:hall}
{\rm \cite{Hall35}}
A bipartite graph $G$ with $V(G)=[X,Y]$ has an $X$-matching if and only if for every subset $W\subset X$ we have $|N(W)|\ge |W|$.
\end{thm}

Perhaps the most interesting class of Kneser graphs is that of {\em odd graphs}, $O_k=K_{2k+1,k}$. The dissociation number of odd graph $O_k$ is by Proposition~\ref{prp:lowerBound} bounded below  by ${2k \choose k}$. Proposition~\ref{prp:Kn2} implies that this bound is also an upper bound for $k=2$. In the next result we prove that the bound is the exact value also for $k > 2$.

\begin{thm}\label{thm:odd}

For any $k \geq 2$, $diss(O_k)={2k \choose k}$.
\end{thm}
\begin{proof}
By Proposition~\ref{prp:lowerBound}, $diss(O_k) \geq {2k \choose k}$.

Let $S$ be a maximum dissociation set of $O_k$. Let $D=\{x \in V(O_k); 2k+1 \notin x\}$ and $U=V(O_k) \setminus D$, that is, $U=\{x \in V(O_k); 2k+1 \in x\}$. Note that $O_k[D]$ is isomorphic to $K_{2k,k}$ which is in turn isomorphic to $\frac{1}{2}{2k \choose k} K_2$. On the other hand, $U$ is a center $I(2k+1)$ of $K_{2k+1,k}$, hence an independent set.

If $S \cap D = \emptyset$, then $|S| \leq |U|={2k \choose k-1} \leq {2k \choose k}$ and the proof is done. If $S \cap U = \emptyset$, then $|S| \leq |D|={2k \choose k}$ which also completes the proof. Thus it remains to consider the case when $D \cap S \neq \emptyset$ and $L=U \cap S \neq \emptyset$. Set $\ell=|L|$. Since $|D|={2k \choose k}$ and $|S \cap U|=\ell$, it suffices to prove that at least $\ell$ vertices from $D$ are not contained in $S$.

Let $E$ be the set of edges having one endvertex in $L$ and the other endvertex in $N_{O_k}(L)$, where $N_{O_k}(L)$ is a subset of $D$. Now, $|E|=(k+1)|L|$, since any $u \in L$ has exactly $k+1$ neighbors in $D$. On the other hand, any $x \in N_{O_k}(L)$ has exactly $k$ neighbors in $U$ and hence $|E| \leq k |N_{O_k}(L)|$. Therefore $|N_{O_k}(L)| \geq \frac{k+1}{k}|L| > |L|$. The same argument applies for any subset of $L$; that is, if $L'\subset L$, then $|N_{O_k}(L')| \geq \frac{k+1}{k}|L'| > |L'|$. Thus, by Theorem~\ref{thm:hall}, there is an $L$-matching $M=\{u_1x_1,u_2x_2,\ldots , u_{\ell}x_{\ell}\}$ in a bipartite graph $G=(L \cup N_{O_k}(L),E)$, where $u_i \in L$ and $x_i \in N_{O_k}(L)$ for any $i \in [\ell]$. Let $[M',M'']$ be a partition of $M$, where $M'$ is the set of edges in $M$ with both endvertices contained in $S$. Since $L \subseteq S$, exactly one endvertex of each edge in $M''$ (that is, the endvertex from $L$) is contained in $S$. Denote by $Z$ (resp.\ $A'$) the set of endvertices in $L$ (resp.\ $N_{O_k}(L)$) of edges in $M'$ and let $W$ (resp.\ $B'$) be the set of endvertices of edges in $M''$ that are contained in $L$ (resp.\ $N_{O_k}(L)$). Furthermore, let $A''=(N_{O_k}(L) \cap S) \setminus A'$ and $B''=(N_{O_k}(L) \setminus S) \setminus B'$. This definitions directly imply that $[A',A'',B',B'']$ is a partition of $N_{O_k}(L)$. Note that $A' \cup A''$ consists exactly of the vertices of $N_{O_k}(L)$ that are contained in $S$, and $B' \cup B''$ contains the vertices from $N_{O_k}(L)$ not contained in $S$.

To complete the proof, we count the number of edges between $L$ and $B' \cup B''$. Denote the set of those edges by $E'$. Since $Z \cup A' \cup A''$ is a subset of a dissociation set $S$ and $u_ix_i \in M$ is an edge in $O_k[S]$, $x_i$ is the only neighbor of $u_i \in Z$ that is contained in $S$. Hence all other $k$ neighbors of $u_i$ in $D$ are from $B' \cup B''$. Since $W \cup A' \cup A'' \subseteq S$, any vertex $u_i \in W$ has at most one neighbor in $A''$ and all other $k$ neighbors of $u_i$ in $D$ are from $B' \cup B''$. Thus $|E'| \geq k (|Z|+|W|)=k|L|$. Since each vertex $x \in B' \cup B''$ has exactly $k$ neighbors in $U$ and as $L$ is a subset of $U$, we get $|E'| \leq k(|B'|+|B''|)$. Consequently $|B'|+|B''| \geq |L|$.  Hence $B' \cup B'' \subseteq D \setminus S$ is a set of at least $\ell$ vertices in $D$ that are not contained in $S.$
\qed        
\end{proof}

\section{Concluding remarks}
 
In this paper, we found the dissociation number of several families of Kneser graphs $K_{n,k}$. This includes the cases $k\in \{2,3\}$ and $n=2k+1$. As proved in Theorem~\ref{thm:large-n}, when $n$ is large enough with respect to $k$, then the dissociation number equals the independence number of the corresponding Kneser graph. The point when this happens for a given $k$, the integer $n_0$, is open, and we give a lower bound for it. Two problems posed in Section~\ref{sec:dissociation} are related to $n_0$. Establishing exact values of $diss(K_{2k+r,k})$, where $r$ is a small integer greater than $1$, is another challenge.  

The dissociation number is dual invariant to the $3$-path vertex cover number (as is the independence number dual to the ($2$-path) vertex cover number). A natural problem is to consider the $m$-path vertex cover number of Kneser graphs, for any given $m>3$. An alternative extension of the problem studied in this paper is the following. Since a dissociation set induces a subgraph with maximum degree at most $1$, it would be interesting to find the largest size of a subset of vertices in $K_{n,k}$ that induces a subgraph with maximum degree $\Delta$, for any given $\Delta\ge 2$. 

\section*{Acknowledgements}
B.B. and T.D. acknowledge the financial support from the Slovenian Research Agency (research core funding No.\ P1-0297 and project grants No.\ J$1$-$9109$ and N$1$-$0095$). 



\begin{thebibliography}{99}


\bibitem{bcl-04} R.~Boliac, K.~Cameron, V.V.~Lozin, On computing the dissociation number and the induced matching number of  bipartite graphs,
{\it Ars Combin.} {\bf 72} (2004) 241--253.

\bibitem{bs-16} C.~Brause, I.~Schiermeyer, Kernelization of the 3-path vertex cover problem. {\it Discrete Math.} {\bf 339} (2016) 1935--1939. 

\bibitem{bhr-17} B.~Bre\v sar, B.L.~Hartnell, D.F.~Rall, Uniformly dissociated graphs, {\it Ars Math.\ Contemp.} {\bf 13} (2017) 293--306.

\bibitem{bjk-13} B.~Bre\v sar, M.~Jakovac, J.~Katreni\v c, G.~Semani\v sin, A.~Taranenko, On the vertex $k$-path cover, {\it Discrete Appl.\ Math.} {\bf 161}  (2013) 1943--1949.

\bibitem{bkk-11} B.~Bre\v sar, F.~Kardo\v s, J.~Katreni\v c, G.~Semani\v sin, Minimum $k$-path vertex cover, {\it Discrete Appl.\ Math.}  {\bf 159}  (2011) 1189--1195.

\bibitem{bkss}  B. Bre\v{s}ar, R. Krivo\v{s}-Bellu\v{s}, G. Semani\v{s}in, P. \v{S}parl, On the weighted k-path vertex cover problem,  {\it Discrete Appl.\ Math.} {\bf 177} (2014) 14--18.

\bibitem{Chen2003} Y.\ Chen, Triangle-free Hamiltonian Kneser graphs, {\it J.\ Combin.\ Theory (B)} {\bf 89} (2003) 1--16.

\bibitem{ekr-61} P.~Erd\H{o}s, C.~Ko, R.~Rado, Intersection theorem for system of finite sets, \textit{Quart.\ J.\ Math.} \textbf{12} (1961) 313--318.

\bibitem{grip-21} L.N.~Grippo, A.~Pastine, P.~Torres, M.~Valencia-Pabon, J.C.~Vera, On the $P_3$-Hull Number of Kneser Graphs, {\it Electron.\ J.\ Combin.} {\bf 28} (2021) P3.32, 9 pp.

\bibitem{Hall35} P.\ Hall, On Representatives of Subsets, \textit{J.\ London Math.\ Soc.} \textbf{10} (1935) 26--30.

\bibitem{hw-2003} C.\ Hartman, D.B.\ West, Covering designs and domination in Kneser graphs; unpublished manuscript, 2003.



\bibitem{iz-1993} J.~Ivan\v{c}o, B.~Zelinka, Domination in Kneser graphs, 
\textit{Math.\ Bohem.} \textbf{118} (1993) 147--152.


\bibitem{jt-13} M.~Jakovac, A.~Taranenko, On the $k$-path vertex cover of some graph products, {\it Discrete Math.} {\bf 313}  (2013) 94--100.


\bibitem{kks-11} F.~Kardo\v s, J.~Katreni\v c, I.~Schiermeyer, On computing the minimum 3-path vertex cover and dissociation number of graphs,
{\it Theoret.\ Comput.\ Sci.}  {\bf 412}  (2011) 7009--7017.

\bibitem{kat-72} G.O.H.~Katona, A simple proof of the Erd\H{o}s-Chao Ko-Rado theorem, {\it J.\ Combin.\ Theory (B)} {\bf 13} (1972) 183--184.


\bibitem{liyawa-17} Y. Li, Z. Yang, W. Wang,
Complexity and algorithms for the connected vertex cover problem in 4-regular graphs, {\it Appl.\ Math.\ Comput.} {\bf 301} (2017) 107--114.


\bibitem{lov-78} L.~Lov\'{a}sz, Kneser's Conjecture, Chromatic Numbers and Homotopy, \textit{J. Combin. Theory (A)} \textbf{25} (1978) 319--324.


\bibitem{mat-04} J.~Matou\v sek, A Combinatorial Proof of Kneser's Conjecture, \textit{Combinatorica} \textbf{24} (2004) 163--170.


\bibitem{odf-11} Y.~Orlovich, A.~Dolgui, G.~Finke, V.~Gordon, F.~Werner, The complexity of dissociation set problems in graphs, {\it Discrete Appl. Math.}
{\bf 159}  (2011) 1352--1366.



\bibitem{osx-2015} P.~R.~J.~\"{O}sterg\r{a}rd, Z.~Shao, X.~Xu, Bounds on the Domination Number of Kneser Graphs, \textit{Ars Math.\ Contemp.} \textbf{9} (2015) 197--205.


\bibitem{py-82} C.H.~Papadimitriou, M.~Yannakakis, The complexity of restricted spanning tree problems, {\it J. Assoc. Comput. Mach.} {\bf 29}  (1982)  285--309.

\bibitem{Si2004} I.B.\ Shields, Hamilton cycle heuristics in hard graphs, North Carolina State University, 2004.


\bibitem{vpv-05} M.~Valencia-Pabon, J.C.~Vera, On the diameter of Kneser graphs, \textit{Discrete Math.} \textbf{305} (2005) 383--385.



\end{thebibliography}
\end{document}